\documentclass[12pt]{amsart}
\usepackage{amsmath}
\usepackage{enumerate}
\usepackage{amssymb}
\usepackage{amsbsy}
\usepackage{amsfonts}
\theoremstyle{plain}

\newtheorem{prop}{Proposition}[section]
\newtheorem{lem}[prop]{Lemma}
\newtheorem{thr}[prop]{Theorem}
\newtheorem{cor}[prop]{Corollary}
\theoremstyle{definition}

\newtheorem{rem}[prop]{Remark}
\newtheorem{dfn}[prop]{Definition}

\newtheorem*{re*}{Remark}
\newtheorem*{ex*}{Example}
\numberwithin{equation}{section}

\newcommand{\sm}{\begin{smallmatrix}}
\newcommand{\esm}{\end{smallmatrix}}

\newfont{\FieldFont}{msbm10 scaled\magstep1}

\newcommand{\mat}{\begin{pmatrix}}
\newcommand{\emat}{\end{pmatrix}}

\def\Ei{{\rm Ei}}

\begin{document}

\title
{Twisted traces of singular moduli of weakly holomorphic modular functions}

\author{D. Choi}
\address{School of Liberal Arts and Sciences, Korea Aerospace University, 200-1,
Hwajeon-dong, Goyang, Gyeonggi 412-791, Korea}
\email{choija@postech.ac.kr}


\begin{abstract}
Zagier proved that the generating series for the traces of singular moduli is a \textit{weakly holomorphic} modular form of weight $3/2$ on $\Gamma_0(4)$. Bruinier and Funke extended the results of Zagier to modular curves of arbitrary genus. Zagier also showed that the twisted traces of singular moduli are generated by a weakly holomorphic modular form of weight $3/2$. In this paper, we study the extension of Zagier's result for the twisted traces of singular moduli to congruence subgroups $\Gamma_0(N)$. As an application, we study congruences for the twisted traces of singular moduli of weakly holomorphic modular functions on $\Gamma_0(N)$.
\end{abstract}

\maketitle

\section{\bf Introduction}
Let $j(z)$ be the usual $j$-invariant function defined for $z$ in the complex
upper half plane $\mathbb{H}$ by $j(z)= q^{-1} + 744 + 196884 q + \cdots, $ where $q=e(z)=e^{2 \pi i z}$. The function $J(z)=j(z)-744$ is the normalized Hauptmodul for the group $\Gamma(1)=PSL_2(\mathbb Z)$. For a positive integer $D$ congruent to 0 or 3 modulo 4,
 denote by ${\mathcal Q}_D$ the set of positive definite integral binary quadratic
 forms $$Q(x,y)=[a,b,c]=ax^2+bxy+cy^2$$ with discriminant
 $-D=b^2-4ac$.  The group $\Gamma(1)$ acts on ${\mathcal Q}_D$ by
 $Q \circ \left( \sm \alpha & \beta \\  \gamma&\delta  \esm\right)=Q(\alpha x+\beta y, \gamma x+ \delta y)$. For each $Q\in {\mathcal Q}_D$ let
 $$z_Q=\frac{-b+i\sqrt{D}}{2a},$$ the corresponding CM point in
 $\mathbb{H}$ and $\Gamma(1)_Q$ denote the stabilizer of $Q$ in $\Gamma(1)$. The values
of $j$ or other modular functions at CM points $z_Q$ are
known as \textit{singular moduli}, and they play important roles in number theory.
For example, if $-D$ is a negative fundamental discriminant, then
$j(z_Q)$ generates the Hilbert class field of $\mathbb{Q}(\sqrt{-D})$ (see \cite{Cox}). We define the trace of singular moduli of index $D$ by
\begin{equation}\label{zt}\textbf{t}_{J}(D)=\sum_{Q\in {\mathcal Q}_D/\Gamma(1)}
 \frac{1}{|{\Gamma}(1)_Q|} J(z_Q).\end{equation}
In [23, Theorem 1], Zagier proved that the generating series for the traces of singular moduli
 \begin{equation}\label{z}g(z):=q^{-1}-2-\mathop{\sum_{D>0}}_{D\equiv 0,3 (\textrm{mod}\ 4)}
 \textbf{t}_{J}(D)q^D=q^{-1}-2+248q^3-492q^4+\cdots\end{equation}
is a weakly holomorphic modular form of weight $3/2$ on $\Gamma_0(4)$, that is, holomorphic on $\mathbb{H}$ and
meromorphic at each cusp. The results of Zagier were extenend in \cite{Kim1} and \cite{Kim2} to congruence
subgroups $\Gamma_0(N)$ of genus zero with prime levels $N$.

Bruinier and Funke \cite{BF} generalized the results of Zagier to the traces of singular moduli of
modular functions on congruence subgroups of arbitrary genus. They proved that the generating
series for the traces of CM values of a weakly holomorphic modular function on a modular curve
of arbitrary genus is given by the holomorphic part of a harmonic weak Maass form of weight
$3/2$. In \cite{Zagier}, Zagier defined the twisted traces of singular moduli for $J(z)$ and proved that
the generating series for the twisted traces of singular moduli is also a weakly holomorphic
modular form of weight $3/2$. In this paper, using the method of Bruinier and Funke \cite{BF}, we
study modularity of the twisted traces of CM values of weakly holomorphic modular functions
and their congruences.

Following the definition of Zagier \cite{Zagier}, we define the twisted traces of a weakly holomorphic
modular function on $\Gamma_0(N)$. For a positive integer $N$, let $\mathcal{Q}_{D,N}$ be the set of quadratic forms
$Q\in \mathcal Q_{D}$ such that $a \equiv  0 \pmod{N}.$ We note that $-D$ is congruent to a square modulo $4N$
and the group $\Gamma_0(N)$ acts on $Q_{D,N}$ with finitely many orbits, where the action of $\Gamma_0(N)$ is
defined as above. Let $\Gamma_0(N)_Q$ be the stabilizers of $Q$ in $\Gamma_0(N)$. Let $\Delta \in \mathbb{Z}$ be a fundamental
discriminant and $r \in \mathbb{Z}$ such that $\Delta \equiv r^2 \pmod{4N}$. Following the definition in \cite{GKZ}, we define
a generalized genus character for $[Na, b, c] \in  \mathcal{Q}_{D,N}$ as follows:
$$\chi_{\Delta}(X)=\chi_{\Delta}([Na,b,c])=\left\{
\begin{array}{ll}
  \left(\frac{\Delta}{n}\right), &
  \begin{array}{l}
  \text{if } \Delta|b^2-4Nac \text{ and } (b^2-4Nac)/\Delta  \text{ is a square }\\
    \text{modulo } 4N \text{ and } \gcd(a,b,c,\Delta)=1,
  \end{array}\\
  0, &\text{ otherwise}.
\end{array}
\right.$$
Here $n$ is any integer prime to $\Delta$ represented by one of the quadratic forms $[N_1a,b,N_2c]$ with
$N_1N_2 = n$ and $N_1,N_2 > 0$ (see 1.2 in \cite{GKZ}). If $f$ is a weakly holomorphic modular function on
$\Gamma_0(N)$, then the twisted trace of $f$ of positive index $D$ is defined by
\begin{equation}\label{first}
\textbf{\bf t}_f(\chi_{\Delta};D) =
 \underset{Q\in\mathcal{Q}_{D,N}/\Gamma_0(N)}{\sum}\chi_{\Delta}(Q)\cdot
  \frac{f (z_Q)}{|\overline{\Gamma_0(N)}_Q|} ,
\end{equation}
With these notations, we state our main theorem.

\begin{thr}\label{main1} Suppose that $N$ is a positive integer, and that $f(z)$ is
a weakly holomorphic modular function on $\Gamma_0(N)$. If $r$ is a sufficiently large integer, then
for each positive odd integer $t$ the function
$$\sum_{\sm m>0 \\ \left(\frac{rm}{t}\right)=-1 \esm}{\bf t}_f(\chi_{\Delta};rm)q^m$$
is a weakly holomorphic modular form of weight $3/2$ on $\Gamma_1( 4 r t^2 N)$. Here,
$\left(\frac{m}{t}\right)$ denotes the Jacobi symbol.
\end{thr}

As an application, we study congruence properties for the twisted traces of CM values of
weakly holomorphic modular functions on $\Gamma_0(N)$. Ahlgren and Ono  \cite{A-O} studied divisibility
of the traces of singular moduli in terms of the factorization of primes in imaginary quadratic
fields. For example, they proved that for each positive integer $\nu$, a positive proportion of primes
$r$ has the property that $\textbf{t}_{J}(r^3n)\equiv 0 \pmod{p^{\nu}}$ for every positive integer $n$ coprime to $r$ such
that $p$ is inert or ramified in $\mathbb{Q}\left(\sqrt{-nr}\right)$. This result was extended in \cite{T} and \cite{CK} to the traces of
singular moduli of a weakly holomorphic modular function on $\Gamma_0(N)$ for any integer $N$. Here,
 $\Gamma_0^*(N)$ denotes the group extension of $\Gamma_0(N)$ by the group of Atkin-Lehner involutions $W_p$ for
all primes $p\mid N$. We obtain analogues of the results in \cite{A-O}, \cite{T} and \cite{CK} for the twisted traces of
CM values of weakly holomorphic modular functions.

\begin{cor}\label{main2}
Suppose that $p$ is an odd prime, and $N$, $p \nmid N$, and $t$ are a positive odd integer. Let $K$ be an algebraic number field. Suppose
that $f(\tau)=\sum a(n)q^n\in K((q))$ is a weakly holomorphic
modular function on $\Gamma_0(N)$. Then there
exists an integer $\Omega$ such that if $m$ is sufficiently large,
then for each positive integer $\nu$, a positive proportion of
primes $r\equiv -1 \pmod{4 t^2 Np^\nu}$ have the property that
$$\Omega{\bf t}_f(\chi_{\Delta};r^3p^m n)\equiv 0 \pmod{p^\nu}$$
for all $n$ such that $\gcd(n,rpN)=1$ and $\left(\frac{r^3p^m n}{t}\right)=-1$.
\end{cor}

This paper is organized as follows. In section 2 we recall basic facts on
real quadratic spaces and modular curves, and then define the twisted traces of CM values of a weakly holomorphic modular function on $\Gamma_0(N)$. In
section 3, to prove the main theorems, we define a theta kernel and a theta lift for a weakly holomorphic modular functions on $\Gamma_0(N)$. In section 4 and 5 we give the proofs of the main theorems.

 \section{\bf Preliminaries}\label{BFW}

 For basic facts on rational quadratic spaces and modular curves, we refer to \cite{BF} and \cite[Section 2]{CK} and follow notations in \cite{BF}. We consider a quadratic space $(V,q)$ over $\mathbb Q$ of signature $(1,2)$ given by $$V(\mathbb Q):=\{X\in M_2 (\mathbb Q)|\mathrm{tr}(X)=0\}$$
 with the associated quadratic form $q(X):=\det (X)$ and the bilinear form $(X,Y):=-{\rm tr}(XY)$.
 The group $SL_2(\mathbb Q)$ acts on $V$ by conjugation:
  $$ g\cdot X := g X g^{-1}$$
  for $X\in V$ and $g\in SL_2 (\mathbb Q)$. This orthogonal transformation gives rise to
  an isomorphism $G:=\mathrm{Spin}(V) \simeq SL_2$. We write $D:=G(\mathbb{R})/K$ for the associated
 orthogonal symmetric space, where $K=SO(2)$. Then, we have $D \simeq \mathbb{H}$, the upper half plane.
  We may regard $D$ as the space of positive lines in $V (\mathbb R)$, that
  is,
${D = \{ \text{span}(X) \subset V (\mathbb R) \, | \,
(X,X)>0
 \}}$ and can give the following identification of $D$
 with $\mathbb{H}$.  We pick as a base point of $D$ the line spanned by  $\mat 0&1\\-1&0 \emat$
 so that $K=SO(2)$ is its stabilizer in $G(\mathbb R)$. For $z=x+iy \in \mathbb{H}$, we choose $g_z\in
  G(\mathbb R)$ such that $g_z i= z$, where the action is the usual linear fractional transformation on
  $\mathbb{H}$.  We now have the isomorphism
  $\mathbb{H} \to D$ which assigns $z \in \mathbb{H}$ the positive line in $D$ spanned by
$$ X(z):=g_z \cdot \mat 0&1\\-1&0 \emat
     = \frac 1y \mat -\frac 12 (z+\bar{z}) & z\bar{z} \\
     -1& \frac 12 (z+\bar{z}) \emat. $$
Note that $q(X(z))=1$ and $g \cdot X(z)=X(gz)$ for $g\in
  G(\mathbb R)$.
  \par
To define CM points in $D$, we need the following set-up.
First, let $L\subset V(\mathbb{Q})$ be an even lattice of full
rank and write
  $L^\#$ for the dual lattice of $L$. If $\Gamma$ denotes a
  congruence subgroup of $\text{Spin}(L)$ which preserves $L$
  and acts trivially on the discriminant group $L^\#/L$, then the attached locally symmetric space
  $M:=\Gamma\backslash D$  is a modular curve, i.e., non-compact, as our space $V$ is isotropic over $\mathbb{Q}$.
The set $\text{Iso}(V)$ of all isotropic lines in $V$ corresponds
to $P^1 (\mathbb Q)=\mathbb Q \cup \{\infty\}$,
  the set of cusps of $G(\mathbb Q),$ via the bijective map
  $ \psi: P^1(\mathbb Q) \to \text{Iso}(V),$ which is defined by
   ${\psi((\alpha:\beta))=\text{span} \mat -\alpha \beta & \alpha^2 \\
   -\beta^2 & \alpha\beta \emat \in \text{Iso}(V). }$
  As $\psi$ commutes with the $G(\mathbb{Q})$-actions, that is,
  $\psi (g(\alpha : \beta))=g\cdot\psi ((\alpha:\beta))$ for $g\in G(\mathbb Q)$, the cusps
  of $M$, i.e., the $\Gamma$-classes of $P^1(\mathbb Q)$, can be
  identified with the $\Gamma$-classes of $\text{Iso}(V)$. In particular, the cusp
  $\infty\in P^1(\mathbb Q)$ is mapped to the isotropic line
  $\ell_0$ which is spanned by $X_0:=\mat 0&1\\0&0
  \emat$. We orient all lines $\ell\in \text{Iso}(V)$ by regarding
  $\sigma_\ell \cdot X_0$ as a positively oriented basis vector of
  $\ell$, where $\sigma_\ell \in SL_2(\mathbb Z)$ such that $\sigma_\ell \cdot\ell_0 = \ell$.
  For each isotropic line $\ell \in \text{Iso}(V)$, there exist positive rational numbers
  $\alpha_\ell$ and $\beta_\ell$ such that $\sigma_\ell^{-1} \Gamma_\ell \sigma_\ell =
    \left\{ \pm \mat 1 & k \alpha_\ell \\ 0& 1 \emat \,\Big|\, k\in \mathbb Z \right\},$
    where $\Gamma_\ell$ denotes the stabilizer of
the line $\ell$, and $\mat
  0&\beta_\ell\\0&0 \emat$ is a primitive element of
  $\ell_0 \cap \sigma_\ell^{-1} L$, respectively.
  Finally, we write $\varepsilon_\ell=\alpha_\ell/\beta_\ell$.
  Note that $\alpha_\ell$ is the width of the cusp $\ell$ of a congruence subgroup of $SL_2(\mathbb Z)$
  and the quantities
  $\alpha_\ell, \beta_\ell,$ and $\varepsilon_\ell$ only depend on the
  $\Gamma$-class of $\ell$.

 We denote the
  space of weakly holomorphic modular forms of weight $k$ on $\Gamma$ by $M_{k}^! (\Gamma).$
   Let us define CM points as, for $X\in V(\mathbb Q)$ of positive norm, i.e.,
  $q(X)>0$,
  $$ D_X = \text{span}(X) \in D.$$
  Note that the corresponding point in $\mathbb{H}$ satisfies a
  quadratic equation over $\mathbb{Q}$.  Since the stabilizer
$G_X$ of $X$ in $G(\mathbb{R})$
  is isomorphic to $SO(2)$ which is compact, $\Gamma_X=G_X\cap
  \Gamma$ is finite.  For $m\in \mathbb Q_{>0}$ and $h\in L^\#$, the group $\Gamma$ acts on
  $$ L_{h,m}=\{X\in L+h\,|\, q(X)=m \}$$
  with finitely many orbits. We define the {\it Heegner divisor}
  of discriminant $m$ on $M$ by
  $$ Z(h,m)= \sum_{X\in \Gamma\backslash L_{h,m}}
  \frac{1}{|\overline{\Gamma_X}|} D_X.$$
  On the other hand, for a vector $X\in V(\mathbb{Q})$ of negative norm, we define a
  geodesic $c_X$ in $D$ by
  $$ c_X=\{ z\in D\,|\, z \bot X \}. $$
  We note from \cite[Lemma 3.6]{Fu} that the case $X^\bot \subset V(\mathbb{Q})$
  is split over $\mathbb{Q}$ is equivalent to $q(X)\in -(\mathbb Q ^\times)^2$.
  In that case, the stabilizer $\overline{\Gamma}_X$ is
  trivial, the quotient $c(X):=\Gamma_X\backslash c_X$ in $M$ is an infinite geodesic, and
   $X$ is orthogonal to the two
  isotropic lines $\ell_X=\text{span}(Y)$ and
  $\tilde{\ell}_X=\text{span}(\tilde{Y})$, with $Y$ and $\tilde{Y}$
  positively oriented. We say $\ell_X$ is the line associated to
  $X$ if the triple $(X, Y, \tilde{Y})$ is a positively oriented
  basis for $V$, and we write $X\sim \ell_X$. Note
  $\tilde{\ell}_X=\ell_{-X}.$

  If  $m\in \mathbb{Q}_{>0}$ and $X\in L_{h,-m^2}$, then we can choose the orientation of $V$ such that
  $$\sigma^{-1}_{\ell}X=\left(
                         \begin{array}{cc}
                           m & r \\
                           0 & -m \\
                         \end{array}
                       \right)$$
  for some $r \in \mathbb{Q}$. The geodesic $c_X$ is given in $D\simeq \mathbb{H}$ by
  $$c_X=\sigma_{\ell_X}\{z\in \mathbb{H} \; |\; \Re(z)=-r/2m\}.$$ The number $-r/2m$ is called
  as the real part of the infinite geodesic $c(X)$ and denoted by $Re(c(X))$. We define
  $$\langle f, c(X) \rangle=-\sum_{n<0}a_{\ell_X}(n)e^{2 \pi i Re(c(X))n}-\sum_{n<0}a_{\ell_{-X}}(n)e^{2 \pi i Re(c(-X))n}.$$

From now on, we define the twisted traces of CM values of a weakly holomorphic modular function on $\Gamma_0(N)$. Let $N$ be a positive integer. Suppose that
$$L=\left\{ \left(\begin{matrix}
  b & 2c \\
  2aN & -b
\end{matrix}\right); a,b,c \in \mathbb{Z} \right\}$$
and $\Gamma=\Gamma_0(N)$. Let $\Delta \in \mathbb{Z}$ be a fundamental discriminant and
$r \in \mathbb{Z}$ such that $\Delta \equiv r \pmod{4N}$. Following the definition in \cite{GKZ}, we define a generalized genus character for $X=\left(\begin{smallmatrix}
  b & 2c \\
  2aN & -b \\
\end{smallmatrix}\right)\in L$ as follows:
$$\chi_{\Delta}(X)=\chi_{\Delta}([Na,b,c]).
$$
With these notations we define the twisted trace of a weakly holomorphic modular function $f$ on $\Gamma$.
\begin{dfn}\label{trace}
Suppose that $f(z)$ is a weakly holomorphic modular function on $\Gamma$. For a general
genus character $\chi_{\Delta}$ we define the twisted trace ${\bf t}_f(\chi_{\Delta};m)$ as follows:
\begin{enumerate}
\item If $m>0$ and $m \in \mathbb{Q}_{>0}$, then
\begin{equation}\label{2.1.1}
{\bf t}_f(\chi_{\Delta};m):=\sum_{X\in
\Gamma\backslash L_{0,m}}
    \frac{\chi_{\Delta}(X)}{|\overline{\Gamma_X}|} f(D_X).
\end{equation}
\item  If $m=0$, or $m<0$ such that $m\not \in - (\mathbb{Q}^{\times})^2$, then
\begin{equation}\label{2.1.2}
{\bf t}_f(\chi_{\Delta};m):=0.
\end{equation}
\item If $m<0$ and $m \in - (\mathbb{Q}^{\times})^2$, then
\begin{equation}\label{2.1.3}
{\bf t}_f(\chi_{\Delta};m):=\sum_{X\in
\Gamma\backslash L_{0,m}}\chi_{\Delta}(X)\langle f,c(X)\rangle.
\end{equation}
\end{enumerate}
\end{dfn}
\begin{rem}
Note that the definition of twisted traces for positive index $m$ in (\ref{2.1.1}) is the same as the definition in (\ref{first}).
\end{rem}

\section{Twisted theta kernels}
Suppose that $N$ is a positive integer and $\Gamma=\Gamma_0(N)$.
In \cite{K}, Kudla constructed a Green function $\xi^{0}$ associated
to a Poincare dual form $\varphi^0(X,z)$ for the Heegner point
$D_X$.  We recall the construction of $\xi^0$. Let
$$\Ei(w)=\int_{-\infty}^{w} \frac{e^t}{t}dt,$$ where the path of
integration lies in the along the positive real axis (see
\cite{A-S}). For $X \in V(\mathbb{R})$, $X\neq0$, we define
\begin{equation}\label{xi}
\xi^0(X,z)=-\Ei(-2 \pi R(X,z)),
\end{equation}
where $$R(X,z)=\frac{1}{2}(X,X(z))^2-(X,X).$$ It is known that $\Ei(\xi^0(X,z))$ is
a smooth function on $D\setminus D_X$. For $q(X)>0$, the function
$\xi^0(X,z)$ has logarithmic growth at the point $D_X$, while it is
smooth on $D$ if $q(X) \leq 0$. Moreover, if $X \neq 0$, then away
from the point $D_X$
$$\frac{1}{2 \pi i}\cdot
 \overline{\partial} \partial \xi^0(X,z)=\varphi(X,z).$$ For $\tau=u+iv \in \mathbb{H}$, let $$\varphi(X,\tau,z)=e^{2 \pi i
q(X)\tau}\varphi^0 (\sqrt{v}X,z).$$
We define a theta kernel $\theta_{\Delta}$ by
\begin{align}
\theta_{\Delta}(\tau, z, \varphi)&=\sum_{X\in \; L} \chi_{\Delta}(X)\varphi(X,\tau,z).
\end{align}
Since $\chi_{\Delta}$ is invariant under the action of $\Gamma_0(N)$ and $$\varphi^{0}(g.\sqrt{v}X,gz)=\varphi^{0}(\sqrt{v}X,z),$$
we have for $g\in \Gamma_0(N)$
$$\theta_{\Delta}(\tau, g z, \varphi)=\theta_{\Delta}(\tau, z, \varphi).$$

\begin{prop}\label{kernel} The theta kernel $\theta_{\Delta}(\tau, z ,\varphi)$
 is a non-holomorphic modular form of weight $3/2$ with values in
$\Omega^{1,1}(M)$ on a congruence subgroup $\Gamma_1(4 N)$. For each
cusp $\ell$ we have
$$\theta_{\Delta}(\tau, \sigma_{\ell}z, \varphi)=O(e^{-cy^2}) \text{ as } y
\rightarrow \infty,$$ uniformly in $x$, for some constant $C>0$.
\end{prop}
\begin{proof}
Take $$L_{\Delta}:=\left\{
\Delta\left(\begin{matrix}
  b & 2c \\
  2aN & -b
\end{matrix}\right); a,b,c \in \mathbb{Z} \right\}.
$$
For $h \in L_{\Delta}$ let
$$\theta_h(\tau, z, \varphi)=\sum_{X\in L_{\Delta}+h} \varphi(X,\tau,z).$$
Since $\chi_{\Delta}(X)$ depends on $X \in L$ modulo $L_{\Delta}$, we have
\begin{equation}\label{theta-chi}
\theta_{\Delta}(\tau, z, \varphi)=\sum_{h\in L/L_{\Delta}}\chi_{\Delta}(h)\sum_{X\in L_{\Delta}} \varphi(X,\tau,z)=\sum_{h\in L/L_{\Delta}}\chi_{\Delta}(h)\theta_h(\tau, z, \varphi).
\end{equation}
It is known that $\theta_h(\tau, z, \varphi)$ is a non-holomorphic modular form of weight $3/2$ with values in $\Omega^{1,1}(M)$ for congruence subgroup $\Gamma(\Delta N)$, and that
for each cusp $\ell$ we have
$$\theta_h(\tau, \sigma_{\ell}z, \varphi)=O(e^{-cy^2}) \text{ as } y
\rightarrow \infty,$$ uniformly in $x$, for some constant $C>0$ (see \cite{Fu}, Proposition 4.1).

Note that if $\Delta \nmid q(X)$, then $\chi_{\Delta}(X)=0$. This implies that
$$\theta_{\Delta}(\tau, z, \varphi)=\sum_{n \in \mathbb{Z}}\sum_{q(X)=n\Delta}\chi_{\Delta}(X)\varphi^{0}(\sqrt{v}X,z)e^{2 \pi i \Delta n \tau}.$$ Let $\Theta(\tau):=1+2\sum_{n=1}^{\infty}q^{n^2}$. Note that the function $\Theta(\tau)$ is
a modular form of weight $1/2$ on $\Gamma_0(4)$.
For a positive integer $m$ and a function $f(\tau)$ on $\mathbb{H}$, we define operators $U_m$ and $V_m$ by
$$f(\tau)|U_m=\frac{1}{m}\sum_{j=0}^{m-1}f\left(\frac{\tau+j}{m}\right) \text{ and } f(\tau)|V_{m}=f(m\tau).$$
Then $$\theta_{\Delta}(\tau, z, \varphi)|U_{\Delta}=\sum_{n \in \mathbb{Z}}\sum_{q(X)=n\Delta}\chi_{\Delta}(X)\varphi^{0}(\sqrt{v/\Delta}X,z)e^{2 \pi i  n \tau}$$
(see Chapter I in \cite{O} for details of operators $U_m$ and $V_m$).Thus, we have
$$\Theta(\Delta\tau)\theta_{\Delta}(\tau, z, \varphi)|U_{\Delta}|V_{\Delta}=\Theta(\Delta\tau)\theta_{\Delta}(\tau, z, \varphi),$$
Following the argument of Lemma 4 in \cite{Li}, we have that
$\Theta(\tau)\theta(\tau, z, \varphi)$ is a non-holomorphic modular form
of weight $2$ on $\Gamma(4 N)$.
Note that $\Theta(\tau)$ is nowhere-vanishing on $\mathbb{H}$ and
$$\theta_{\Delta}(\tau+1, z, \varphi)=\theta_{\Delta}(\tau, z, \varphi).$$ Thus, we complete the proof.
\end{proof}
We define a theta lift of $f$ by
\begin{equation}\label{t-lift}
I_{\Delta}(\tau,f)=\int_{M}f(z)\theta_{\Delta}(\tau,z ,\varphi).
\end{equation}
Proposition \ref{kernel} implies the convergence of the integral (\ref{t-lift}). Thus, we have
$I_{\Delta}(\tau,f)$ is a (in general non-holomorphic) modular form of weight $3/2$ on $\Gamma_1(N)$.

\section{Proof of Theorem \ref{main1}}
Suppose that $N$ is a positive integer and $\Gamma=\Gamma_0(N)$.
By determining the Fourier expansion of $I_{\Delta}(\tau,f)$, we prove that the generating series for the twisted
traces of CM values of a weakly holomorphic modular function on $\Gamma_0(N)$ is given by the holomorphic part of a harmonic weak Maass form of
weight 3/2. Note that we have
\begin{align*}
I_{\Delta}(\tau,f)&=\int_{M}f(z)\sum_{X\in L} \chi_{\Delta}(X)\varphi^{0}(\sqrt{v}X,z)e^{2 \pi i q(X)\tau}\\
&=\sum_{m\in \mathbb{Z}} \int_{M}\sum_{X\in L_{0,m}}\chi_{\Delta}(X)f(z)\varphi^{0}(\sqrt{v}X,z)e^{2 \pi i q(X)\tau},
\end{align*}
where  $$L_{0,m}=\{X\in L\,|\, q(X)=m \}.$$
If $m\neq0$, then, since $\Gamma\setminus L_{0,m}$ is finite, we have
\begin{align*}
\int_{M}\sum_{X\in L_{0,m}}&\chi_{\Delta}(X)f(z)\varphi^{0}(\sqrt{v}X,z)e^{2 \pi i q(X)\tau}\\
&=\sum_{X\in \Gamma\setminus L_{0,m}} \chi_{\Delta}(X) \int_{M} \sum_{\gamma \in \Gamma_X\setminus \Gamma}f(z) \varphi^{0}(\sqrt{v}X,\gamma z)e^{2 \pi i q(X)\tau}.
\end{align*}
In the following proposition, we determine the $m$th Fourier coefficient of $I_{\Delta}(\tau,f)$
for $m \neq 0$.
\begin{prop}{\cite{BF}}\label{4.1}
Let $X\in L_{0,m}$ and $\Gamma=\Gamma_0(N)$. Then we have the followings:
\begin{enumerate}
\item If $m>0$ and  $X \in L_{0,m}$, then
$$\int_{M}\sum_{\Gamma_X\backslash\Gamma}f(z)\varphi^0(\sqrt{v}X,\gamma z)=\frac{1}{|\overline{\Gamma_0(N)_X}|}f(D_X).$$
\item If $m<0$ and $m\not \in -(\mathbb{Q}^{\times})^2$ and $X\in L_{0,m}$, then
$$\int_{M}\sum_{\Gamma_X\backslash\Gamma}f(z)\varphi^0(\sqrt{v}X,\gamma z)\in L^1(M)$$ and
$$\int_{M}\sum_{\Gamma_X\backslash\Gamma}f(z)\varphi^0(\sqrt{v}X,\gamma z)=0.$$
\item If $m<0$ and $m \in -(\mathbb{Q}^{\times})^2$, then
$$\int_{M}\sum_{\Gamma_X\backslash\Gamma}f(z)\varphi^0(\sqrt{v}X,\gamma z)\in L^1(M)$$ and $$\int_{M}\sum_{\Gamma_X\backslash\Gamma}f(z)\varphi^0(\sqrt{v}X,\gamma z)=(a_{\ell_{X}}(0)+a_{\ell_{-X}}(0))\frac{1}{8\pi\sqrt{v}m}\beta(4\pi m)+\langle f,c(X)\rangle.$$
\end{enumerate}
\end{prop}
Recall that $\alpha_\ell$ is the width of the cusp $\ell$
of $\Gamma$, and that
  $\sigma_\ell \in SL_2(\mathbb Z)$ transforms the infinite cusp to $\ell$.
Then $f$ has a Fourier expansion at the cusp
$\ell$ of the form
\begin{equation}\label{fourier}f(\sigma_\ell z) = \sum_{n\in \frac{1}{\alpha_\ell}\mathbb Z}
  a_\ell (n) e(nz)\end{equation}
  with $a_\ell (n) = 0$ for $n\ll 0$.
\begin{prop}\label{4.2} If $m=0$, then we have
\begin{align*}
\int_{M} \sum_{X \in \; L_{0,0}}f(z)
\chi_{\Delta}(X)\varphi^{0}(\sqrt{v}X,z)=0.
\end{align*}
\end{prop}
\begin{proof}
Let $Q_{\ell}=e(\sigma^{-1}_{\ell}z/\alpha_{\ell})$ and $D_{1/T}=\left\{w\in\mathbb{C}\; | \; 0<|W|<\frac{1}{2 \pi T}\right\}$ for $T>0$. We truncate $M$ by setting
$$M_{T}=M-\coprod_{\ell\setminus Iso(V)}Q^{-1}_{\ell}D_{1/T}.$$ The regularized integral
$\int_{M}^{reg} \sum_{X \in \; L_{0,0}}f(z)
\chi_{\Delta}(X)\varphi^{0}(\sqrt{v}X,z)$ is defined by
$$\int_{M}^{reg} \sum_{X \in \; L_{0,0}}f(z)
\chi_{\Delta}(X)\varphi^{0}(\sqrt{v}X, z)=\lim_{T\rightarrow0}\int_{M_T} \sum_{X \in \; L_{0,0}}f(z)
\chi_{\Delta}(X)\varphi^{0}(\sqrt{v}X, z).$$

Let $X_{\ell}$ be the primitive positive oriented vector in $L \cap \ell$. Note that $\chi_{\Delta}(\left(\begin{smallmatrix}
0 & 0\\
0 & 0\\
\end{smallmatrix}\right))=0$. Thus, we have by Stokes Theorem
\begin{align*}
\int_{M} \sum_{X \in \; L_{0,0}}f(z)
\chi_{\Delta}(X)&\varphi^{0}(\sqrt{v}X,z)=\int^{reg}_{M}f(z)\sum_{X\in L_{0,0}}\chi_{\Delta}(X)\varphi^0(\sqrt{v}X,z)\\
&=\int^{reg}_{M}f(z)\sum_{\sm \ell\in \Gamma \setminus Iso(V)\\
L_{0,0}\cap \ell \neq 0\esm}\sum_{ X\in \ell \; \cap L_{0,0}}\chi_{\Delta}(X)
\sum_{\gamma \in \Gamma_{\ell}\setminus \Gamma}\varphi^0(\sqrt{v}\gamma^{-1}X,z)\\
&=\sum_{\sm \ell\in \Gamma \setminus Iso(V)\\
L_{0,0}\cap \ell \neq 0\esm}\sum_{\gamma \in \Gamma_{\ell}\setminus \Gamma}
\int^{reg}_{M}f(z)\sum_{n=-\infty}^{\infty}\chi_{\Delta}(n X_{\ell})\varphi^0(n\sqrt{v}X_{\ell},\gamma z)\\
&=\frac{1}{2 \pi i} \sum_{\sm \ell\in \Gamma \setminus Iso(V)\\
L_{0,0}\cap \ell \neq 0\esm}\lim_{T\rightarrow0} \int_{\partial M_T}f(z)\sum_{\gamma \in \Gamma_{\ell}\setminus \Gamma} \sum_{n=-\infty}^{\infty}\chi_{\Delta}(n X_{\ell})\partial \xi^0(n\sqrt{v}X_{\ell},\gamma z).
\end{align*}
Let $X^0_{\ell}=\left(\sm 0 & \beta_{\ell} \\ 0 & 0 \esm \right)$.
We have by (\ref{xi})
$$\partial \xi^0(\sqrt{v}X^0_{\ell},g z)=\frac{-i}{(cz+d)^2 Im(gz)}e^{-\pi v r^2/Im(gz)^2}dz$$
for $g=\left(\sm a&b \\ c&d \esm \right) \in SL_2(\mathbb{R})$.
This implies that there is $\delta>0$ such that
$$|\partial \xi^0(\sqrt{v}X^0_{\ell},g z)|\ll e^{-\delta y^2}dz$$
for all $g=\left(\sm a&b \\ c&d \esm \right) \in SL_2(\mathbb{R})$ with $c \neq 0$, uniformly for $y>1$.
Thus, we have
\begin{align*}
&\frac{1}{2 \pi i} \sum_{\sm \ell\in \Gamma \setminus Iso(V)\\
L_{0,0}\cap \ell \neq 0\esm}\lim_{T\rightarrow0} \int_{\partial M_T}f(z)\sum_{\gamma \in \Gamma_{\ell}\setminus \Gamma} \sum_{n=-\infty}^{\infty}\chi_{\Delta}(n X_{\ell})\partial \xi^0(n\sqrt{v}X_{\ell},\gamma z)\\
&=\frac{-1}{2 \pi i} \sum_{\sm \ell\in \Gamma \setminus Iso(V)\\
L_{0,0}\cap \ell \neq 0\esm}\sum_{\ell' \in \Gamma \setminus Iso(V)}\lim_{T\rightarrow0} \int_{z=iT}^{\alpha_{\ell'}+iT}f(\sigma_{\ell'} z)\sum_{\gamma \in \Gamma_{\ell}\setminus \Gamma} \sum_{n=-\infty}^{\infty}\chi_{\Delta}(n X_{\ell})\partial \xi^0(n\sqrt{v}X^0_{\ell},\sigma_{\ell}^{-1}\gamma\sigma_{\ell'}z)\\
&=\frac{-1}{2 \pi i} \sum_{\sm \ell\in \Gamma \setminus Iso(V)\\
L_{0,0}\cap \ell \neq 0\esm}\lim_{T\rightarrow0} \int_{z=iT}^{\alpha_{\ell}+iT}f(\sigma_{\ell'} z) \sum_{n=-\infty}^{\infty}\chi_{\Delta}(n X_{\ell})\partial \xi^0(n\sqrt{v}X^0_{\ell},z)\\
&=\frac{1}{2 \pi } \sum_{\sm \ell\in \Gamma \setminus Iso(V)\\
L_{0,0}\cap \ell \neq 0\esm}\lim_{T\rightarrow0} \int_{z=iT}^{\alpha_{\ell}+iT}f(\sigma_{\ell'} z) \sum_{n=-\infty}^{\infty}\chi_{\Delta}(n X_{\ell})
\frac{1}{y}e^{-\pi v (n  \beta_{\ell})^2/y^2}dx\\
&=\frac{1}{2 \pi } \sum_{\sm \ell\in \Gamma \setminus Iso(V)\\
L_{0,0}\cap \ell \neq 0\esm}\lim_{T\rightarrow0}\frac{\alpha_{\ell}}{2 \pi} a_{\ell}(0) \sum_{n=-\infty}^{\infty}\chi_{\Delta}(n X_{\ell})
\frac{1}{T}e^{-\pi v (n  \beta_{\ell})^2/T^2}dx.
\end{align*}
Note that for a fixed $X_{\ell}$ we have $\chi_{\Delta}(nX_{\ell})$ is a Dirichlet character (see I.2 in \cite{GKZ}). Let $R$ be the conductor of $\chi_{\Delta}(nX_{\ell})$ and
$$ g(\chi_{\Delta}(nX_{\ell}))=\sum_{n\mod{R}} \chi_{\Delta}(n)e^{2 \pi i n/R}.$$
Using twisted Poisson summation formula (see the formula (1.10) in \cite{B}), we have
\begin{align*}
&\int^{reg}_{M}f(z)\sum_{X\in L_{0,0}}\chi_{\Delta}(X)\varphi^0(\sqrt{v}X,z)\\
&=\frac{\chi_{\Delta}(-X)g(\chi_{\Delta}(nX_{\ell}))}{R}\cdot\frac{a_{\ell}(0)(\alpha_{\ell}/\beta_{\ell})}
{2\pi \sqrt{v}}\lim_{T\rightarrow\infty}\sum_{w=-\infty}^{\infty}\chi_{\Delta}(wX)e^{-\pi w^2 T^2/(v R^2 \beta_{\ell}^2)}=0.
\end{align*}
This completes the proof.
\end{proof}

Using Proposition 4.7 in \cite{BF}, we immediately obtain by (\ref{theta-chi}) that ${\bf t}_f(\chi_{\Delta},-m^2)$ is zero for large $m > 0$.
\begin{prop}\label{4.3}
Suppose that $f(z)$ is a weakly holomorphic modular function on $\Gamma_0(N)$ having the Fourier expansion as in
(\ref{fourier}). Then
$${\bf t}_f(\chi_{\Delta},-m^2)=0 \;\;\; \text{ for } m \gg 0.$$
\end{prop}

Proposition \ref{4.1}, \ref{4.2} and \ref{4.3} immediately give the
Fourier expansion of $I_{\Delta}(\tau,f)$.
\begin{thr}\label{main3}
  Let $f\in M_0^! (\Gamma_0(N))$ with Fourier expansion as in (\ref{fourier}) and $\tau=u+iv\in \mathbb{H}$.
  If the constant coefficient of $f$ at each cusp of $M$ vanishes, then $I_{\Delta}(\tau,f)$ is a weakly
  holomorphic modular form of weight 3/2 for
  $\Gamma_1(N)$. The Fourier expansion of $I_{\Delta}(\tau,f)$ is given by
  $$I_{\Delta}(\tau,f)= \sum_{\sm m\in \mathbb{Z} \\ m\gg -\infty \esm}
  {\bf t}_f(\chi_{\Delta};m) q^m.$$
   If the constant coefficient
  of $f$ does not vanish, then $I_{\Delta}(\tau, f)$ is non-holomorphic, and
  in the Fourier expansion the following terms occur in addition:
  $$ \sum_{m>0}\sum_{X\in \Gamma_0(N)\backslash L_{0,-m^2}}
   \frac{ a_{\ell_X} (0) + a_{\ell_{-X}} (0)}{8\pi \sqrt v m}
   \chi_{\Delta}(X)\beta(4\pi v m^2) q^{-m^2} .$$
 \end{thr}
Now we prove Theorem \ref{main1}.
\begin{proof}[Proof of Theorem \ref{main1}]
Suppose that $I_{\Delta}(\tau,f)$ has the Fourier expansion of the form $$I_{\Delta}(\tau,f)=\sum_{m}a_m(v)e^{2 \pi i m  u}.$$
Let  $$F_t(\tau)=\frac{1}{2}\sum_{ m}\left(\frac{m}{t}\right)^2 a_m(v)e^{2 \pi i m  u}-
\frac{1}{2}\left(\frac{-1}{t}\right)\sum_{ m}\left(\frac{m}{t}\right)a_m(v)e^{2 \pi i m  u}.$$
Theorem \ref{main3} implies
that $F_t(\tau)$ is a weakly holomorphic modular form of weight $3/2$ on $\Gamma_1(4t^2N)$,
 which has the Fourier expansion of the form
\begin{equation}\label{Ft}
F_t(\tau)=\sum_{\sm m> m_0 \\ \left(\frac{m}{t}\right)=-1 \esm}\textbf{t}_{f}(\chi_{\Delta}; m)q^m.
\end{equation}
Note that we can take $m_0$ independent of $t$. Thus, if $r$ is sufficiently large, then
$$F_t(\tau)|U_r=2\sum_{\sm m>0 \\ \left(\frac{m}{t}\right)=-1 \esm}\textbf{t}_{f}(\chi_{\Delta}; rm)q^m.$$ This completes the proof.
\end{proof}

\section{Proof of Theorem \ref{main2}}
We begin by stating the following lemma.
\begin{lem}\cite[Theorem 1.1]{T}\label{treneer-thm}
Suppose that $p$ is an odd prime, and that $k$ and $m$ are
integers with $k$ odd. Let $N$ be a positive integer with $4|N$
and $p\nmid N$. Let $g(\tau)=\sum a(n)q^n \in
M_{\frac{k}{2}}^{!}(\Gamma_0(N))\cap\mathcal{O}_K((q))$, where
$\mathcal{O}_K$ denotes the ring of integers of an algebraic
number field $K$. If $m$ is sufficiently large, then for each
positive integer $\nu$, a positive proportion of primes $r\equiv
-1 \pmod{Np^\nu}$ have the property that
$$a(r^3p^m n)\equiv 0 \pmod{p^\nu}$$
for all $n$ relatively prime to $rp$.
\end{lem}
To use Lemma \ref{treneer-thm} in our case, we
have to show that traces of singular moduli are algebraic numbers.
\begin{lem}\label{algebraic number}
If $f(\tau)\in M^{!}_{0}(\Gamma_0(N))\cap K((q))$, then
$\textbf{t}_{f}(\chi_{\Delta},m)$ is an algebraic number for every integer $m$.
\end{lem}
\begin{proof}
Assume that $m$ is a positive integer.
Following the argument of \cite[Lemma 5.2]{CK}, we have
that $f(\tau_0)$ is an algebraic number for a CM point $\tau_0$. Thus,
$\textbf{t}_{f}(\chi_{\Delta},m)$ is also an algebraic number  by (\ref{2.1.1}).
In case when $m$ is not positive, we can easily see that
$\textbf{t}_{f}(\chi_{\Delta},m)$ is an algebraic number as well from
(\ref{2.1.2}) and (\ref{2.1.3}).
\end{proof}
Now we prove Theorem \ref{main2}.
\begin{proof}[Proof of Theorem \ref{main2}]
Let
\begin{equation*}
F_t(\tau)=2\sum_{\sm m> m_t \\ \left(\frac{m}{t}\right)=-1
\esm}\textbf{t}_{f}(\chi_{\Delta}; m)q^m.
\end{equation*} From (\ref{Ft}) we have $F_t(\tau)$
is a weakly holomorphic modular form of weight $3/2$ on
$\Gamma_1(4t^4N)$. Since
$$\Delta(\tau):=q\prod_{n=1}^{\infty}(1-q^n)^{24}$$ is the cusp
form of weight 12 on $SL_2(\mathbb{Z})$,  for sufficiently large
$m$, $\Theta(\tau)\Delta^m(\tau) F_t(\tau)$ is a cusp form, and
its Fourier coefficients are algebraic numbers.  Hence, any
$\mathbb{Z}$-module generated by the Fourier coefficients of
$\Theta(\tau)\Delta^m(\tau)F_t(\tau)$ is finitely generated
\cite[Theorem 3.52]{Shimura}. Thus, there exists an integer $\mu$
such that
\begin{equation}\label{mu}\mu F_t(\tau)\in
  M^{!}_{\frac{3}{2}}(\Gamma_1(4t^2N))\cap\mathcal{O}_K((q)),\end{equation}
  where $K$ is the field generated by Fourier coefficients of
  $F_t(\tau)$.
Note that Lemma \ref{treneer-thm} can be extended to $\Gamma_1(N)$ (See
\cite{Br-O}). Thus, from Lemma \ref{treneer-thm} and (\ref{mu}) we
complete the proof.
\end{proof}

\section*{acknowledgement}
This work was supported by the Korea Research
Foundation Grant funded by the Korean Government
(KRF-2008-331-C00005). The author wishes to express his gratitude to KIAS for its support
through Associate membership program.

\end{document}